\documentclass[amstex,12pt, amssymb]{article}

\usepackage{mathtext}
\usepackage[cp1251]{inputenc}
\usepackage[T2A]{fontenc}
\usepackage[dvips]{graphicx}
\usepackage{amsmath}
\usepackage{amssymb}
\usepackage{amsxtra}
\usepackage{latexsym}
\usepackage{ifthen}

\newcounter{lemma}[section]

\newcounter{corollary}[section]

\newcounter{remark}[section]

\newcounter{theorem}[section]

\newcounter{proposition}[section]

\numberwithin{equation}{section}

\begin{document}

\markboth{\centerline{E.~SEVOST'YANOV}} {\centerline{ON ESTIMATES OF
MODULI OF FAMILIES OF SURFACES}}

\def\cc{\setcounter{equation}{0}
\setcounter{figure}{0}\setcounter{table}{0}}

\overfullrule=0pt


\author{{E. SEVOST'YANOV}\\}

\title{
{\bf ON ESTIMATES OF MODULI OF FAMILIES OF SURFACES}}

\date{\today}
\maketitle

\begin{abstract}
A class of space mappings of finite distortion with $N$ and
$N^{\,-1}$ Luzin properties with respect to $k$-measured area is
investigated. It is proved that, mappings mentioned above satisfy
upper and lower inequalities for families of $k$-measures surfaces.
\end{abstract}


\section{Introduction } As known, a
class of maps with finite area distortion consists of mappings,
which distort distance between points in a finite number of time and
satisfy $N$ and $N^{\,-1}$-Luzin properties with respect to
$k$-dimensional area (see \cite[Ch.~10]{MRSY}). For such mappings,
some modular inequalities with respect to families of surfaces are
established in \cite{MRSY}. At the present paper we strengthen above
results by considering more wide classes of surfaces, for which such
inequalities still hold. The order of modulus $p$ is assumed to
satisfy the inequality $p\geqslant 1.$ Give some definitions.

\medskip
Everywhere below, $D$ is a domain in ${\Bbb R}^n,$ $n\ge 2,$ $m$ is
the Lebesgue measure in ${\Bbb R}^n,$ $m(A)$ the Lebesgue measure of
a measurable set $A\subset {\Bbb R}^n,$ $m_1$ is the linear Lebesgue
measure in ${\Bbb R}.$ Recall that a mapping $f:X\rightarrow Y$
between measurable spaces $\left(X,\Sigma,\mu\right)$ and
$\left(X^{\,\prime},\Sigma^{\,\prime},\mu^{\,\prime}\right)$ is said
to have the {\it $N$-property} if $\mu^{\,\prime}(f(S))=0$ whenever
$\mu(S)=0.$ Similarly, $f$ has the {\it $N^{\,-1}$-property} if
$\mu(S)=0$ whenever $\mu^{\,\prime}(f(S))=0.$
\medskip

\medskip
For $x\in E\subset {\Bbb R}^n$ and a mapping $\varphi:E \rightarrow
{\Bbb R}^n,$ we set
$$L(x,\varphi)=\limsup\limits_{y\rightarrow x, y
\in E}\,\frac{|\varphi(x)-\varphi(y)|}{|y-x|}\,,\quad
l(x,\varphi)=\liminf\limits_{y\rightarrow x, y\in E}\,
\frac{|\varphi(x)-\varphi(y)|}{|y-x|}\,.$$

\medskip
Let $D$ be a domain in ${\Bbb R}^n.$ Following to
\cite[разд.~8.3]{MRSY}, a mapping $f:D\rightarrow {\Bbb R}^n$ is
called a mappings with {\it finite metric distortion,} write $f\in
FMD,$ if $f$ has the Luzin $N$-property and
$$0<l(x, f)\leqslant L (x, f)<\infty$$
for almost all $x\in D.$

\medskip
Let $\omega$ be an open set in  $\overline{{\Bbb R}^k}:={\Bbb
R}^k\cup\{\infty\},$ $k=1,\ldots,n-1.$ A (continuous) mapping
$S:\omega\rightarrow D$ is called a $k-$dimensional surface $S$ in
${\Bbb R}^n.$ Sometimes we call the image $S(\omega)\subset {\Bbb
R}^n$ by the surface $S$, too. The number of preimages $N(y, S)={\rm
card}\,S^{-1}(y)={\rm card}\,\{x\in\omega:S(x)=y\},\ y\in{\Bbb R}^n$
is said to be a {\it multiplicity function} of the surface $S$ at a
point $y\in {\Bbb R}^n.$ In other words, $N(S,y)$ means the
multiplicity of covering of the point $y$ by the surface $S.$

\medskip
For a given Borel set $B\subset {\Bbb R}^n$ (or, more generally, for
a measurable set $B$ with respect to the $k$-dimensional Hausdorff
measure ${\mathcal H}^k$), the $k$-dimensional Hausdorff area of $B$
in ${\Bbb R}^n$ associated with the surface $S$ is determined by
\begin{equation*}
\mathcal A_S(B)=\mathcal A_S(B)=\int\limits_B N(S, y)\, d {\mathcal
H}^k y\,.
\end{equation*}
If $\rho:\,{\Bbb R}^n\rightarrow [0,\infty]$ is a Borel function,
the integral of $\rho$ over $S$ is defined as
$$\int\limits_S
\rho\,d{\mathcal{A}}:=\int\limits_{{\Bbb R}^n}\rho(y)\,N(y,
S)\,d{\mathcal H}^ky\,.$$
Let $\Gamma$ be a family of $k$-dimensional surfaces $S$ in ${\Bbb
R}^n$, $1\le k\le n-1$ (curves for $k=1$). Given $p\geqslant 1,$ the
$p$-module of $\Gamma$ is defined by
$$M_p(\Gamma)=\inf\limits_{\rho\in{\rm adm}\,\Gamma}
\int\limits_{{\Bbb R}^n}\rho^p(x)\,dm(x)\,,$$
where the infimum is taken over all Borel measurable functions
$\rho\geqslant 0$ and such that
\begin{equation}\label{eq8.2.6}
\int\limits_S\rho^k\,d{\mathcal{A}}\geqslant 1\end{equation}
for every $S\in \Gamma.$ We call each such $\rho$ an {\it admissible
function} for $\Gamma$ ($\rho\in {\rm adm}\, \Gamma$). The
$n$-module $M_n(\Gamma)$ will be denoted by $M(\Gamma).$ The modulus
is itself an outer measure on the collection of all families
$\Gamma$ of $k$-dimensional surfaces (see \cite{Fu}).
\medskip

Following \cite{MRSY}, a metric $\rho$ is said to be {\it
extensively admissible} for $\Gamma$ ($\rho\in {\rm ext}_p{\rm
adm}\, \Gamma$) with respect to $p$-module if $\rho\in {\rm adm}\,
(\Gamma\backslash\Gamma_0)$ such that $M_p(\Gamma_0)=0$ (cf.
\cite{Gol11}). Accordingly, we say that a property $\mathcal P$
holds for almost every $k$-dimensional surface with respect to
$p$-modulus, write $p$-a.e. surface, if $\mathcal P$ holds for all
surfaces except a family of zero $p$-module.

\medskip
A surface $S$ in $D$ is a {\it lifting} of a surface $\widetilde{S}$
under $f:D\rightarrow {\Bbb R}^n,$ if $\widetilde{S}=f(S).$

\medskip
Following \cite[section~10.1]{MRSY}, we say that a mapping
$f:D\rightarrow {\Bbb R}^n$ has {\it $(A_k)$-property} with respect
to $p$-modulus if the two conditions hold:

\medskip
$(A^{(1, p)}_k)$\quad for $p$-a.e. $k$-dimensional surface $S$ in
$D$ the restriction $f|_S$ has $N$-pro\-perty with respect to area;
\medskip

$(A^{(2, p)}_k)$\quad for $p$-a.e. $k$-dimensional surface $S_*$ in
$\widetilde{D}=f(D)$ the restriction $f|_S$ has $N^{\,-1}$-property
for each lifting $S$ of $S_*$ with respect to area.

\medskip
We also say that a mapping $f:D\rightarrow {\Bbb R}^n$ is of {\it
finite area distortion in dimension} $k=1,...,n-1$ with respect to
$p$-modulus, abbr. $f\in FAD_k$ with respect to $p$-modulus, if
$f\in FMD$ and has the $(A_k)$-property  with respect to
$p$-modulus.

\medskip
Let $D$ be a domain in ${\Bbb R}^k,$ $k=1,...,n-1.$ If
$S_1:D\rightarrow {\Bbb R}^n$ is a surface and if $S_2$ is a
restriction of $S_1$ to a subdomain $D_*\subset D,$ we write
$S_2\subset S_1.$ We say that $\Gamma_2$ is {\it minorized} by
$\Gamma_1$ and write $\Gamma_2>\Gamma_1$ if every $S\subset\Gamma_2$
has a subsurface which belongs to $\Gamma_1.$ It is known that
$M_p(\Gamma_1)\geqslant M_p(\Gamma_2)$, see \cite[Theorem~1(c)]{Fu}.

\medskip
Set at points $x\in D$ of differentiability of $f$
$$l\left(f^{\,\prime}(x)\right)\,=\,\min\limits_{h\in {\Bbb
R}^n \backslash \{0\}} \frac {|f^{\,\prime}(x)h|}{|h|}\,, \Vert
f^{\,\prime}(x)\Vert\,=\,\max\limits_{h\in {\Bbb R}^n \backslash
\{0\}} \frac {|f^{\,\prime}(x)h|}{|h|}\,, J(x,f)=det
f^{\,\prime}(x)\,,$$
and define for any $x\in D$ and fixed $p\geqslant 1$
\begin{equation}\label{eq0.1.1A}
K_{I, p}(x,f)\quad =\quad\left\{
\begin{array}{rr}
\frac{|J(x,f)|}{{l\left(f^{\,\prime}(x)\right)}^p}, & J(x,f)\ne 0,\\
1,  &  f^{\,\prime}(x)=0, \\
\infty, & {\rm otherwise}
\end{array}
\right.\,,
\end{equation}
$$K_{O, p}(x,f)\quad =\quad \left\{
\begin{array}{rr}
\frac{\Vert f^\prime(x)\Vert^p}{|J(x,f)|}, & J(x,f)\ne 0,\\
1,  &  f^{\,\prime}(x)=0, \\
\infty, & {\rm otherwise}
\end{array}
\right.\,.$$

\medskip
A mapping $f:D\rightarrow {\Bbb R}^n$ is {\it discrete} if
$f^{-1}(y)$ consists of isolated points for each $y\in{\Bbb R}^n,$
and $f$ is {\it open} if it maps open sets onto open sets. The
notation $f:D\rightarrow {\Bbb R}^n$ assumes that $f$ is continuous.
In what follows, a mapping $f$ is supposed to be orientation
preserving, i.e., the topological index $µ(y, f,G)>0$ for an
arbitrary domain $G\subset D$ such that $\overline{G}\subset D$ and
$y\in f(G)\setminus f(\partial G).$ Let $f:D\rightarrow {\Bbb R}^n$
be a mapping and suppose that there is a domain $G\subset D,$
$\overline{G}\subset D,$ for which $
f^{\,-1}\left(f(x)\right)=\left\{x\right\}.$ Then the quantity
$\mu(f(x), f, G),$ which is referred to as the local topological
index, does not depend on the choice of the domain $G$ and is
denoted by $i(x, f).$

\medskip
The following proposition can be found in \cite[Lemma 8.3]{MRSY}.

\medskip
\begin{lemma}\label{lem2.4}\label{pr1}{\sl\, Let $f:D\rightarrow{\Bbb R}^n$ be differentiable
a.e. in $D$ and have $N$-- and $N^{-1}$--pro\-per\-ties. Then there
is a countable collection of compact sets $C^*_k\subset D$ such that
$m(B_0)=0$ where $B=D\setminus\bigcup\limits_{k=1}\limits^{\infty}
C^*_k$ and $f|_{C^*_k}$ is one--to--one and bi--lipschitz for every
$k=1,2,\ldots ,$ and, moreover, $f$ is differentiable at points of
$C_k^*$ with $J(x,f)\ne 0.$}
\end{lemma}

\section{The analog of the V\"{a}is\"{a}l\"{a} inequality}

\medskip
The following result generalizes the well--known V\"{a}is\"{a}l\"{a}
inequality for the mappings with bounded distortion, see
\cite[Theorem 3.1]{Va$_3$} and \cite[Theorem~9.1, гл.~II]{Ri}. For
mappings with finite distortion see e.g. \cite[Theorems~8.6 and
Lemma~10.2]{MRSY}, \cite[Theorem~4.1]{KO}, \cite[Theorem~7]{HP} and
\cite[Theorem~3.1]{SalSev}.

\medskip
\begin{theorem}\label{th3.1}
{\sl Let $f:D\rightarrow {\Bbb R}^n$ be an open discrete mapping of
finite metric distortion with the $(A^{(2, p)}_k)$-property for some
$p\geqslant k,$ $1\leqslant k\leqslant n-1.$ Let $\Gamma$ be a
family of $k$-measured surfaces $\alpha$ in $D,$ $\Gamma^{\,\prime}$
be a family of $k$-measured surfaces $\beta:\omega\rightarrow f(D),$
and $m$ be a positive integer such that the following is true. For
every curve $\beta\in \Gamma^{\,\prime}$ there are surfaces
$\alpha_1:\omega_1\rightarrow D,$ $\ldots,$
$\alpha_m:\omega_m\rightarrow D$ in $\Gamma$ such that $f\circ
\alpha_j\subset \beta$ for all $j=1,\ldots,m,$ and for every $x\in
D$ and all $z\in \omega_j$ the equality $\alpha_j(z)=x$ holds at
most $i(x,f)$ indices $j.$ Then
%
$$M_p(\Gamma^{\,\prime} )\quad\leqslant\quad \frac{1}{m}\quad\int\limits_D
K_{I, p}\left(x,\,f\right)\cdot \rho^p (x)\ \ dm(x)\qquad \forall
\,\,\rho \in {\rm }\,{\rm adm}\,\Gamma\,.$$ }
\end{theorem}

Here we write $f\circ\alpha\subset \beta$ iff
$\beta:\omega\rightarrow {\Bbb R}^n,$ $\omega$ is open set in ${\Bbb
R}^k,$ and $\beta|_{\omega_1}=f\circ\alpha$ for some open set
$\omega_1\subset\omega.$

\begin{proof} Let $B_0$ and $C_{\nu}^{*}$  be as in Lemma \ref{lem2.4} and $B_f$ is a branch set for $f$ in $D.$
Note that $m(B_f)=0,$ see \cite[Proposition~8.4]{MRSY}. Setting by
induction $B_0\,=\,B\cup B_f,$ $B_1\,=\,C_{1}^{*}\setminus B_f,$
$B_2\,=\,C_{2}^{*}\setminus (B_1\cup B_f)\ldots\,\,,$
$$B_{\nu}\,\,=\,\, C_{\nu}^{*}\setminus
\left(\bigcup\limits_{i=1}^{\nu-1}B_i \cup B_f\right)\,,$$
we obtain a  countable covering of $D$ consisting of mutually
disjoint Borel sets $B_{\nu},$ $\nu=1,2,\ldots,$ with $m(B_0)=0.$
Since $f$ has $N$-property, $m(f(B_0))=0.$ By
\cite[Theorem~9.1]{MRSY}, ${\mathcal A}_{S_*}(f(B_0))=0$ for
$p$-a.e. surface $S_*$ in $f(D)$ and all $S$ for which $f(S)=S_*.$
Thus, by $(A^{(2, p)}_k)$-property
\begin{equation}\label{equ3}
{\mathcal A}_S(B_0)=0
\end{equation}
for $p$-a.e. surface $S_*$ in $f(D)$ and all $S$ with $f(S)=S_*.$ We
show that (\ref{equ3}) holds for $p$-a.e. $S_*\in \Gamma^{\,\prime}$
and every $S\in \Gamma$ such that $f\circ S\subset S_*.$

Denote $\Gamma_1$ a family of surface $S_*\,\in\,\Gamma^{\,\prime}$
for which
\begin{equation}\label{equ4}
{\mathcal A}_S(B_0)>0
\end{equation}
and some surface $S$ with $f\circ S\subset S_*.$ Assume in the
contrary, that $M_p(\Gamma_1)>0.$ Let $\Gamma_2$ be a family of all
subsurfaces $S_{**}$ of $\Gamma_1,$ having a lifting $S,$ for which
(\ref{equ4}) holds. Since $\Gamma_2<\Gamma_1,$ we obtain that
$M_p(\Gamma_2)\geqslant M_p(\Gamma_1)>0.$ We reach a contradiction
to our assumption that $M_p(\Gamma_1)>0.$

Let $\rho\,\in\,{\rm  }\,\,{\rm adm}\,\Gamma.$ Set
$$\rho^*(x)\,=\,\left\{\begin{array}{rr}
\rho(x)/l\left(f^{\prime}(x)\right), &   x\in D\setminus B_0,\\
0,  &  x\in B_0
\end{array}
\right.$$
and
%
$$\widetilde{\rho}(y)\quad=\quad\left(\frac{1}{m}\cdot
\chi_{f\left(D\setminus B_0
\right)}(y)\sup\limits_{C}\sum\limits_{x\,\in\,C}\rho^{*k}(x)\right)^{1/k}\,,$$
and $C$ runs over all subsets of $f^{\,-1}(y)$ in $D\setminus B_0$
such that ${\rm card}\,C\leqslant m.$ Note that
\begin{equation}\label{equ5}
\widetilde{\rho}(y)\quad=\quad\left(\frac{1}{m}\cdot
\sup\sum\limits_{i=1}^s \rho^k_{{\nu}_i}(y)\right)^{1/k}\,,
\end{equation}
where $\sup$ in (\ref{equ5}) is taken over all
$\left\{{\nu}_{i_1},\ldots,{\nu}_{i_s}\right\}$  such that
${\nu}_i\in\,{\Bbb N},$ ${\nu}_i\ne {\nu}_j$ if $i\ne j,$ all $s\le
m$ and
$$
\rho_{\nu}(y)\,=\left\{\begin{array}{rr}
\,\rho^{*}\left(f_{\nu}^{-1}(y)\right), &   y\in f(B_{\nu}),\\
0,  &  y\notin f(B_{\nu})
\end{array}
\right.$$
where $f_{\nu}=f|_{B_{\nu}},$ $\nu=1,2,\ldots\,$ is injective. Thus,
the function  $\widetilde{\rho}(y)$ is Borel, see e.g.
\cite[2.3.2]{Fe}.

\medskip
Denote $$A_{\nu l}:=\{y\in {\Bbb R}^n: \exists\,w\in \omega_l:
\alpha_l(w)\in B_{\nu}, f(\alpha_l(w))=y\}\,.$$
Set
$$
\varphi_{\nu l}(y):=f_{\nu}^{\,-1}(y), \qquad y\in A_{\nu l}\,.$$
Since $f|_{B_{\nu}}$ is a homeomorphism, $\varphi_{\nu l}$ are
well-defined. Moreover, since $\alpha_j(z)=x$ holds for at most
$i(x,f)$ indices $j,$ there are at most $m$ points $\varphi_{\nu_1
1}(y),\ldots, \varphi_{\nu_s s}(y),$ $1\leqslant s\leqslant m,$
which are different, i.e., $\nu_i\ne \nu_j,$ $i\ne j.$ Setting
$\rho^*(\varphi_{\nu l}(y))=0$ for $y\not\in A_{\nu l},$ we observe
that $\rho^*\circ\varphi_{\nu l}:{\Bbb R}^n\rightarrow {\Bbb R}$ is
a Borel function. Now, we obtain from (\ref{equ5}) that
$$\widetilde{\rho}^{\,k}(y)\geqslant \frac{1}{m}\sum\limits_{l=1}^{s}
\rho^{*k}(\varphi_{\nu_l l}(y))\,=\frac{1}{m}
\sum\limits_{l=1}^{m}\sum\limits_{i=1}^{\infty}
\rho^{*k}(\varphi_{\nu l}(y))\,.
$$
Now
%
%
%
$$\int\limits_{\beta}\widetilde{\rho}^{\,k}\,d{\mathcal A_*}=
\int\limits_{{\Bbb R}^n}\widetilde{\rho}^{\,k}\,N(\beta,
y)d{\mathcal H}^ky\quad\geqslant \frac{1}{m}\cdot
\sum\limits_{l=1}^{m}\sum\limits_{\nu=1}^{\infty}\int\limits_{{\Bbb
R}^n}\rho^{*k}(\varphi_{\nu l}(y))N(\beta, y)\,d{\mathcal H}^ky=
$$
\begin{equation}\label{eq1A}
=\frac{1}{m}\cdot
\sum\limits_{l=1}^{m}\sum\limits_{\nu=1}^{\infty}\int\limits_{A_{\nu
l}}\rho^{*k}(f_{\nu}^{\,-1}(y))N(\beta, y)\,d{\mathcal H}^ky\,.
\end{equation}
Setting $B_{\nu l}=\{x\in B_{\nu}: f(x)\in A_{\nu l}\},$ by
\cite[Theorem~3.2.5]{Fe} for $m=k$ we obtain that
$$
\sum\limits_{\nu=1}^{\infty}\int\limits_{A_{\nu
l}}\rho^{*k}(f_{\nu}^{\,-1}(y))N(\beta, y)\,d{\mathcal
H}^ky=\sum\limits_{\nu=1}^{\infty}\int\limits_{B_{\nu
l}}\rho^{*k}(x)N(\beta, f(x))\,J_kf(x) d{\mathcal H}^kx\,\geqslant$$
$$\geqslant \sum\limits_{\nu=1}^{\infty}\int\limits_{B_{\nu
l}}\frac{\rho^k(x)}{l^k\left(f^{\prime}(x)\right)} N(\alpha_l,
x)J_kf(x)\,d{\mathcal H}^kx\,\geqslant
\sum\limits_{\nu=1}^{\infty}\int\limits_{B_{\nu
l}}\rho^k(x)N(\alpha_l, x)d{\mathcal H}^kx=$$
\begin{equation}\label{eq3}
=\sum\limits_{\nu=1}^{\infty}\int\limits_{{\Bbb
R}^n}\rho^k(x)N(\alpha_l, x)\chi_{B_{\nu l}}(x) d{\mathcal H}^kx
=\int\limits_{{\Bbb R}^n}\rho^k(x)N(\alpha_l,
x)\sum\limits_{\nu=1}^{\infty}\chi_{B_{\nu l}}(x) d{\mathcal H}^kx=
\end{equation}
$$=\int\limits_{{\Bbb R}^n}\rho^k(x)N(\alpha_l,
x)d{\mathcal H}^kx\geqslant 1\,.$$
It follows from (\ref{eq1A}) and (\ref{eq3}) that
$\widetilde{\rho}\,\in\,{\rm }\,\,{\rm
adm}\,\Gamma^{\,\prime}\setminus \Gamma_0,$ where $M_p(\Gamma_0)=0.$
By subadditivity of $p$-modulus
\begin{equation}\label{equ6}
M_p\left(\Gamma^{\,\prime}\right)\quad\leqslant\quad\int\limits_{f(D)}\widetilde{\rho}\,^p(y)\,\,dm(y)\,.
\end{equation}
By \cite[Theorem~3.2.5]{Fe} for $m=n$ we obtain that
\begin{equation}\label{equ7}
\int\limits_{B_{\nu}}K_{I,p}(x,\,f)\cdot\rho^p(x)\,\,dm(x)\quad=\quad\int\limits_{f(D)}
\rho^p_{\nu}(y)\,dm(y)\,.
\end{equation}
By H\"{o}lder inequality for series,
\begin{equation}\label{equ8}
\left(\frac{1}{m}\cdot\sum\limits_{i=1}^{s}\rho^k_{{\nu}_i}(y)\right)^{p/k}\quad\leqslant\quad
\frac{1}{m}\cdot \sum\limits_{i=1}^{s}\,\rho^p_{{\nu}_i}(y)
\end{equation}
for each $1 \leqslant s \leqslant m$ and every
$\left\{{\nu}_1,\ldots,{\nu}_s\right\},$ ${\nu}_i\in {\Bbb N},$
${\nu}_i\ne {\nu}_j,$ if $i\ne j.$

Finally, by Lebesgue positive convergence theorem, see Theorem
I.12.3 in \cite{Sa}, we conclude from (\ref{equ6}), (\ref{equ7}) and
(\ref{equ8}) that
%
$$\frac{1}{m}\cdot\int\limits_{D}K_{I, p}(x,\,f)\cdot\rho^p(x)\,\,dm(x)\quad
=\quad\frac{1}{m}\cdot\int\limits_{f\left(D\right)}\,\sum\limits_{\nu=1}^{\infty}
\rho_{\nu}^p(y)\,dm(y)\quad\geqslant$$
%
%
$$\geqslant\quad\frac{1}{m}\cdot\int\limits_{f\left(D\right)}
\sup\limits_{\left\{{\nu}_1,\ldots,{\nu}_s\right\},\, {\nu}_i\in
{\Bbb N},\atop {\nu}_i\ne {\nu}_j, \,\, i\ne j}\sum\limits_{i=1}^s
\rho^p_{{\nu}_i}(y)\,dm(y)\quad=\quad
\int\limits_{f\left(D\right)}\,\widetilde{\rho}^{\,p}(y)\,dm(y)\quad\geqslant\quad
M_p(\Gamma^{\,\prime})\,.$$
The proof is complete.~$\Box$
\end{proof}

\section{On another modular inequality}

The following statement have been proved for $p=n$ in
\cite[Lemma~10.1]{MRSY}.

\medskip
\begin{theorem}\label{lem8.3.1} Let a mapping $f:D\rightarrow{\Bbb R}^n$ be of finite
metric distortion with $(A^{(1, p)}_k)$-property for some
$p\geqslant k,$ $1\leqslant k\leqslant n-1,$ and let a set
$E\subset\Omega$ be Borel.  Then
\begin{equation}\label{eq8.3.2}
M_p(\Gamma) \leqslant \int\limits_{f(E)} K_{I, p}(y, f^{\,-1},
E)\cdot\rho_*^p(y) dm(y)\,,\end{equation}
for every family $\Gamma$ of $k$-dimensional surfaces $S$ in $E$ and
$\varrho_*\in {\rm adm}\,f(\Gamma)$ where
\begin{equation}\label{eq8.3.3}
K_{I, p}(y, f^{\,-1},E)\ =\sum\limits_{x\in E\cap f^{\,-1}(y)}K_{O,
p}(x, f)\,.\end{equation}
\end{theorem}

\begin{proof} By Theorem III.6.6 (iV)  \cite{Sa},
$E=B\cup B_0,$ where $B$ is a set of the class $F_{\sigma},$ and
$m(B_0)=0.$ Consequently, $f(E)$ is measurable by $N$-pro\-per\-ty
of $f.$ Without loss of generality, we may assume that $f(E)$ is a
Borel set and that $\rho_*\equiv 0$ outside of $f(E).$ In other case
we can find a Borel set $G$ such that $f(E)\subset G$ and
$m(G\setminus f(E))=0,$ see (ii) of the Theorem III.6.6 in
\cite{Sa}. Now, a set $f^{-1}(G)$ is  Borel and $E\subset
f^{-1}(G).$ Note that in this case the function
$$
\rho^G_* (y)= \left \{\begin{array}{rr}
\rho_*(y), &  {\rm for }\,\,\, y\in G, \\
0, & y\in \overline{{\Bbb R}^n}\setminus G\end{array} \right.$$
is a Borel function, as well.  Now suppose that $f(E)$ is a Borel
set. Let $B_0$ and $C_k^*,$ $k=1,2,\ldots ,$ be as in Lemma
\ref{lem2.4}. Setting by induction $B_1=C_1^*,$ $B_2=C_2^*\setminus
B_1,\ldots ,$ and
\begin{equation} \label{eq7.3.7y} B_k=C_k^*\setminus
\bigcup\limits_{\nu=1}\limits^{k-1}B_{\nu}\,,
\end{equation}
we obtain a countable covering of $D$ consisting of mutually
disjoint Borel sets $B_k, k=0,1,2,\ldots $ with $m(B_0)=0,$
$B_0=D\setminus \bigcup\limits_{k=1}^{\infty} B_k.$

Note that by 2) in Remark 9.1 in \cite{MRSY} ${\mathcal A}_S(B_0)=0$
for $p$-a.e. $k$-dimensional surface $S$ in $\Omega$ and by $(A^{(1,
p)}_k)$-property ${\mathcal A}_{S_*}(f(B_0))=0,$ where $S_*=f\circ
S$ also for a.e. $k$-dimensional surface $S.$

Given $\rho_*\in {\rm adm}\,f(\Gamma),$ set
\begin{equation}\label{eq8.3.5}\rho(x)\ =\ \left \{\begin{array}{rr}
\rho_*(f(x))\Vert f^{\,\prime}(x)\Vert\ ,&\text{for }\ x\in D\setminus B_0, \\
0\ ,& \text{otherwise}\end{array}\right.\end{equation}
Arguing piecewise on $B_l$, we have by 3.2.20 and 1.7.6 in \cite{Fe}
and Theorem 9.1 in \cite{MRSY}, see also Remark 9.2 in \cite{MRSY},
that
$$
\int\limits_S\rho^k\ d{\mathcal A}=\int\limits_{{\Bbb
R}^n}\rho^k(x)N(S, x) d{\mathcal H}^kx=\sum\limits_{\nu=1}^{\infty}
\int\limits_{B_{\nu}}\rho^k(x)N(S, x) d{\mathcal H}^kx=$$
$$=\sum\limits_{\nu=1}^{\infty}
\int\limits_{B_{\nu}}\rho_*(f(x))\Vert f^{\,\prime}(x)\Vert N(S, x)
d{\mathcal H}^kx=$$$$=\sum\limits_{\nu=1}^{\infty}
\int\limits_{B_{\nu}}\frac{\rho_*(f(x))\Vert
f^{\,\prime}(x)\Vert}{J_kf(x)}\cdot J_kf(x) N(f(S), f(x)) d{\mathcal
H}^kx\geqslant$$$$\geqslant\sum\limits_{\nu=1}^{\infty}
\int\limits_{B_{\nu}}\rho_*(f(x))\cdot J_kf(x) N(f(S), f(x))
d{\mathcal
H}^kx=\sum\limits_{\nu=1}^{\infty}\int\limits_{B_{\nu}}\rho_*(y)\cdot
N(f(S), y) d{\mathcal H}^ky=$$
%
$$=\int\limits_{f(D)}\rho_*(y)\cdot N(f(S), y) d{\mathcal
H}^ky=\int\limits_{S_*}\rho_*^k\ d{\mathcal A}\geqslant 1$$
for a.e. $S\in\Gamma.$

Note that $\rho =\sum\limits_{k=1}\limits^{\infty}\rho_k$, where
$\rho_k = \rho\cdot\chi_{B_k}$ have mutually disjoint supports. By
3.2.5 for $m=n$ in \cite{Fe} we obtain that
$$\int\limits_{f(B_k\cap E)} K_{O, p}\left(f_k^{-1}(y),
f)\cdot\rho_*^p(y\right) dm(y)= \int\limits_{B_k} K_{O, p}(x,
f)\rho_*^p\left(f(x)\right)|J(x, f)|dm(x)=$$
\begin{equation}
\label{eq7.3.7x} =\int\limits_{B_k} \Vert f^{\,\prime}(x)\Vert^p
\rho_*^p\left(f(x)\right)dm(x)=\int\limits_{D}\rho_k^p(x) dm(x)\,,
\end{equation}
where every $f_k=f|_{B_k},$ $k=1,2,\ldots $ is injective by the
construction.

\medskip
Finally, by the Lebesgue positive convergence theorem, see e.g.
Theorem I.12.3 in \cite{Sa}, we conclude from (\ref{eq7.3.7x}) that
$$\int\limits_{f(E)} K_{I,p}(y, f^{-1}, E)\cdot\rho_*^p(y)dm(y)=
\int\limits_{D} \sum\limits_{k=1}\limits^{\infty}\rho_k^p(x)dm(x)
\geqslant M_p(\Gamma)\,.\,\Box$$

\end{proof}

\medskip
{\bf \noindent Evgeny Sevost'yanov} \\
Zhitomir Ivan Franko State University,  \\
40 Bol'shaya Berdichevskaya Str., 10 008  Zhitomir, UKRAINE \\
Phone: +38 -- (066) -- 959 50 34, \\
Email: esevostyanov2009@mail.ru
\end{document}